\documentclass[A4paper,reqno,oneside,11pt]{amsart}
\usepackage{fullpage}
\usepackage{setspace}
\usepackage{amsmath, amssymb, amsfonts, amsthm, bbm, mathtools, mathrsfs}
\usepackage{tikz}
\usepackage{subcaption}
\usepackage{verbatim,enumitem,graphics}
\usepackage{xcolor}
\usepackage[backref=page]{hyperref}
\hypersetup{%
    colorlinks,
    linkcolor={red!50!black},
    citecolor={green!50!black},
    urlcolor={blue!50!black}
}
\usepackage{dsfont}
\usepackage{lineno}
\usepackage[normalem]{ulem}

\allowdisplaybreaks

\theoremstyle{plain}
\newtheorem{theorem}{Theorem}[section]
\newtheorem*{theorem*}{Theorem}
\newtheorem{lemma}[theorem]{Lemma}

\newtheorem{conjecture}[theorem]{Conjecture}

\theoremstyle{definition}

\newtheorem*{remark*}{Remark}

\newcommand{\scC}{\mathscr{C}}

\newcommand{\scF}{\mathscr{F}}

\newcommand{\scK}{\mathscr{K}}
\newcommand{\scL}{\mathscr{L}}

\newcommand{\scP}{\mathscr{P}}

\newcommand{\F}{\mathbb{F}}

\newcommand{\ex}{\operatorname{ex}}

\def\d{\delta}

\def\l{\ell}

\def\ve{\varepsilon}

\begin{document}

\setstretch{1.27}

\title{Two counterexamples to a conjecture about even cycles}

\author{David Conlon}
\address{Department of Mathematics, California Institute of Technology, Pasadena, CA 91125, USA}
\email{dconlon@caltech.edu}

\author{Eion Mulrenin}
\address{Department of Mathematics, Emory University, Atlanta, GA 30322, USA}
\email{eion.mulrenin@emory.edu}

\author{Cosmin Pohoata}
\address{Department of Mathematics, Emory University, Atlanta, GA 30322, USA}
\email{cosmin.pohoata@emory.edu}

\thanks{The first author was supported by NSF Award DMS-2348859 and the third author was supported by NSF Award DMS-2246659.}

\begin{abstract}
A conjecture of Verstra\"ete states that for any fixed $\l < k$ there exists a positive constant $c$ such that any $C_{2k}$-free graph $G$ contains a $C_{2\l}$-free subgraph with at least $c |E(G)|$ edges. 
For $\l = 2$, this conjecture was verified by K\"uhn and Osthus in 2004. 
We identify two counterexamples to this conjecture for $\l = 4$ and $k=5$: the first comes from a recent construction of a dense $C_{10}$-free subgraph of the hypercube and the second from Wenger's construction for extremal $C_{10}$-free graphs.
\end{abstract}

\maketitle


\section{Introduction}

For a family $\scF$ of graphs and a positive integer $n$, the {\it extremal number} $\ex(n,\scF)$ is the maximum number of edges in an $n$-vertex graph which contains no member of $\scF$ as a subgraph.
Trivially, $\ex(n,\scF) \leq \ex(n,F)$ for all $F \in \scF$. Conversely, the {\it compactness conjecture} of Erd\H{o}s and Simonovits~\cite{ES82} asserts that if $\scF$ is a finite family which does not contain a forest or, equivalently, if $\ex(n,F) = \Omega(n^{1+\ve})$ for some $\ve > 0$ and all $F \in \scF$,\footnote{Without this hypothesis, families such as $\scF = \{K_{1,2}, 2K_2\}$, where $2K_2$ is a matching with two edges, give simple counterexamples. See for example ~\cite{Wig} for more details.} then there is some $F \in \scF$ such that $\ex(n,F) = O(\ex(n,\scF))$.

This conjecture is of most interest when the family $\scF$ includes bipartite graphs, since a celebrated result of Erd\H{o}s and Stone~\cite{ES46} asymptotically determines the behavior of the extremal number for all non-bipartite graphs. In particular, one might ask whether the conjecture holds for the family 
\begin{equation*}
    \scC_{2k} := \{C_3, C_4, \dots, C_{2k}\}
\end{equation*}
of all cycles of length at most $2k$, where it is essentially asking if $\ex(n, C_{2k}) = O(\ex(n, \scC_{2k}))$ --- see, for example,~\cite[Conjecture IV]{Verstraete} for an explicit statement of this question as a conjecture. Despite the fact~\cite{BS74} that both of these functions are known to be $O(n^{1 + 1/k})$, this simple question of whether they agree up to a constant has only been resolved for $k \in \{2,3,5\}$, with the corresponding lower bounds in these cases witnessed by the so-called {\it generalized polygons}~\cite{LUW99}.

Inspired by this problem, a number of natural extremal questions about the relationships between graphs avoiding particular even cycles have been studied. 
One of the first results in this vein was a theorem of Gy\H{o}ri~\cite{Gyori}, who proved that every $C_6$-free bipartite graph $G$ contains a $C_4$-free subgraph with at least $\frac 12 |E(G)|$ edges. 
Several years later, K\"uhn and Osthus~\cite{KO05} extended Gy\H{o}ri's result by proving that for all $k \geq 3$, every $C_{2k}$-free bipartite\footnote{By the elementary fact that any graph can be made bipartite by deleting at most half the edges, this theorem remains true for  non-bipartite $G$, but at the cost of an extra factor of $1/2$.} graph $G$ has a $C_4$-free subgraph with at least $\frac{1}{k-1} e(G)$ edges.

In the same paper~\cite{KO05}, K\"uhn and Osthus stated the following conjecture, which they attribute to Verstra\"ete (cf.~\cite[Conjecture VIII]{Verstraete}) and which would easily imply that $\ex(n, C_{2k}) = O(\ex(n, \scC_{2k}))$, 
saying that their theorem should extend to all pairs of even cycles.

\begin{conjecture}
\label{conj: main}
    For all integers $2 \leq \l< k$, there exists a positive constant $c$ such that every $C_{2k}$-free bipartite graph $G$ has a $C_{2\l}$-free subgraph $F$ with $|E(F)| \geq c|E(G)|$.
\end{conjecture}

K\"uhn and Osthus also verified this conjecture for other values of $\ell$ and $k$, including, for any given $\ell$, infinitely many values of $k$. 
Despite this evidence for the conjecture, our first result says that arbitrarily large counterexamples exist for $C_8$ and $C_{10}$.

\begin{theorem}
\label{thm: hypercube}
    For any positive constant $c$ and any positive integer $N_0$, there exists an integer $N \geq N_0$ and a $C_{10}$-free bipartite graph $G$ on $N$ vertices with the property that every subgraph with at least $c|E(G)|$ edges contains a $C_8$.
\end{theorem}

\noindent 
We cannot claim much credit for this construction, as it is a straightforward consequence of a beautiful recent result by Grebennikov and Marciano~\cite{GM25} of a dense $C_{10}$-free subgraph of the hypercube. 
However, as the graphs $G$ in Theorem~\ref{thm: hypercube} have only $\Theta(N \log N)$ edges, it is natural to wonder whether denser counterexamples exist.
Our second result, which is our main contribution, confirms this, showing that there is a family of counterexamples with essentially the maximum possible number of edges. 


\begin{theorem}
\label{thm: wenger}
    For any positive constant $c$ and any positive integer $N_0$, there exists an integer $N \geq N_0$ and a 
    $C_{10}$-free bipartite graph $W$ with $N$ vertices on each side, $N^{6/5}$ edges and the property that every subgraph with at least $c|E(W)|$ edges contains a $C_8$.
\end{theorem}

\noindent
We will take the graphs $W$ to be those from Wenger's construction~\cite{Wenger} for extremal $C_{10}$-free graphs, which we describe in Section~\ref{section: wenger} below.


\section{Proof of Theorem~\ref{thm: hypercube}}
\label{section: hypercube}

Here and throughout, $Q_n$ denotes the $n$-dimensional hypercube with $2^n$ vertices and $n 2^{n-1}$ edges and, for graphs $G$ and $H$, $\ex(G,H)$ denotes the maximum number of edges in an $H$-free subgraph of $G$.
Theorem~\ref{thm: hypercube} is an immediate consequence of the following two theorems.

\begin{theorem}[Grebennikov and Marciano, 2025~\cite{GM25}]
\label{thm: GM25}
    For every positive integer $n$,
    there is a subgraph $F_n$ of $Q_n$ which is $C_{10}$-free and has $|E(F_n)| > 0.024 |E(Q_n)|$.
\end{theorem}

\begin{theorem}[Chung, 1992~\cite{Chung}]
\label{thm: C10}
    $\ex(Q_n,C_8) = o(|E(Q_n)|)$.
\end{theorem}

\begin{proof}[Proof of Theorem~\ref{thm: hypercube}]
    Suppose for the sake of contradiction that there is $c > 0$ such that every $C_{10}$-free bipartite graph $G$ has a $C_8$-free subgraph with at least $c|E(G)|$ edges.
    Taking such a subgraph of the graph $F_n$, which is a subgraph of $Q_n$ and so necessarily bipartite, from Theorem~\ref{thm: GM25} for each $n$, we obtain a sequence of $C_8$-free subgraphs of $Q_n$ containing at least $0.024 c |E(Q_n)|$ edges of $Q_n$, contradicting Theorem~\ref{thm: C10}.
\end{proof}

For the interested reader, we include a brief sketch of the construction of the graphs $F_n$ in Theorem~\ref{thm: GM25}, which was itself inspired by the recent work of Ellis, Ivan, and Leader~\cite{EIL24} on the Tur\'an densities of daisies. The main task is to construct a subgraph of $Q_n$ of positive relative density with no copy of $C_6^-$, the graph obtained by removing an edge from a copy of $C_6$ in $Q_n$. Note that every $C_6^-$ is a path of length five, but not every path of length five is a $C_6^-$. 
This is sufficient because of the inequality
\begin{equation*}
    \ex(Q_n, C_{10}) \geq \frac{1}{3} \cdot \ex^*(Q_n,C_6^-)
\end{equation*}
noted by Axenovich, Martin, and Winter~\cite{AMW24},
where $\ex^*(Q_n,C_6^{-})$ denotes the maximum number of edges in a subgraph of $Q_n$ with no $C_6^-$.
Grebennikov and Marciano construct the required subgraph as follows:
\begin{itemize}
    \item First identify the $r$th level of $Q_n$ for each $0 \leq r \leq n$ with $[n]^{(r)}$ and consider the layers between levels $r-1$ and $r$ for all odd $r$.
    \item Observe that every induced subgraph of a layer that contains a $C_6^-$ also contains a $C_6$, so it is enough to construct a dense $C_6$-free induced subgraph $G_r$ of the $r$th layer for each odd $r$.
    \item Fix a vector $v_0 \in \F_2^r \setminus \{0\}$ and, for each $i \in [n]$, choose a vector $v_i \in \F_2^r \setminus \{0\}$ uniformly at random.
    \item Now let $G_r$ be the induced subgraph between the sets
    \begin{align*}
        & B_r := \{S \in [n]^{(r)}: \{v_i: i \in S\} \text{ forms a basis for } \F_2^r\},\\
        & B_{r-1} := \{S \in [n]^{(r-1)}: \{v_0\} \cup \{v_i: i \in S\} \text{ forms a basis for } \F_2^r\}.
    \end{align*}
    \item Each of these random graphs $G_r$ is dense in the $r$th layer in expectation and is (deterministically) $C_6$-free.
\end{itemize}

As a matter of fact, Chung's result is slightly stronger than what is stated above.
She proved that
\begin{equation*}
    \ex(Q_n, C_8) = O(n^{-1/4} e(Q_n)),
\end{equation*}
obtaining a polylogarithmic saving in terms of the number of vertices.
In the next section, we will prove a stronger power-saving result for $C_8$-free subgraphs of Wenger graphs.


\section{Proof of Theorem~\ref{thm: wenger}}
\label{section: wenger}

We first recall Wenger's construction~\cite{Wenger} for $C_{2k}$-free graphs, following the simplified interpretation given in~\cite{Con21}.
We denote by $W_k(q)$ the bipartite incidence graph on $\scP \cup \scL$, where $\scP = \F_q^k$ 
and $\scL$ is the set of affine lines in $\F_q^k$ with directions of the form $(1,a,\dots,a^{k-1})$ for $a \in \F_q$.
More precisely, for each $a \in \F_q$, we define $\scL_a$ to be the set of all lines of the form 
\begin{equation*}
    \{v + t \cdot (1,a,a^2,\dots,a^{k-1}): t \in \F_q\}
\end{equation*}
for some $v \in F_q^k$ and then we set\footnote{More generally, we can pick any set of lines with the property that any $k$ of the directions determined by the parallel classes are linearly independent; here, the choice of directions is the moment curve $\{(1, a, a^2, \dots, a^{k-1}): a \in \F_q\}$.}
\begin{equation*}
    \scL = \bigcup_{a \in \F_q} \scL_a.
\end{equation*}
It can be shown via a (simple) case analysis that, provided $k \geq \l$, $W_k(q)$ is $C_{2\l}$-free for $\l=2,3,5$ 
and contains copies of $C_{2\l}$ when $\l \notin \{2,3,5\}$ --- see~\cite[Theorem 1]{Con21}.
Theorem~\ref{thm: wenger} thus follows immediately by applying the following more general theorem to the $C_{10}$-free Wenger graphs $W_5(q)$. 
Recall that for graphs $G$ and $H$, $\ex(G,H)$ denotes the maximum number of edges in an $H$-free subgraph of $G$.

\begin{theorem}
\label{thm: main}
    For a prime power $q$, let $\scL_1, \scL_2, \dots, \scL_q$ be $q$ parallel classes of lines in $\F_q^5$, each of size $q^4$, and let $G_\scL(q)$ be the bipartite incidence graph between $\scP = \F_q^5$ and $\scL = \bigcup_{i \in [q]} \scL_i$. 
    Then, for every $\d > 0$, there exists $q_0(\delta)$ such that, for all prime powers $q \geq q_0(\d)$,
    every subgraph $H \subseteq G_\scL(q)$ with at least $\d q^6$ edges contains a copy of $C_8$.
    In fact, as $q \to \infty$,
    \begin{equation*}
        \ex \left( G_\scL(q),C_8 \right)= O(q^{23/4}).
    \end{equation*}
\end{theorem}

\begin{proof}
    Let $q$ be a large prime power, fix parallel classes of lines $\scL_1, \dots, \scL_q$ in $\F_q^5$ and let $\scL = \bigcup_{i \in [q]} \scL_i$ and $G = G_\scL(q)$.
    Throughout the proof, any asymptotic notation will be used in the regime where $q \to \infty$.

    We begin with the following lemma, which says that the intersection graph on $\scL$ has a very rigid structure.

    \begin{lemma}
    \label{lem: intersection-graph}
        For two distinct parallel classes $\scL_i, \scL_j \subseteq \scL$, the intersection graph between $\scL_i$ and $\scL_j$ is a vertex-disjoint union of $q^3$ copies of
        $K_{q,q}$, indexed by the affine $2$-planes in $\F_q^5$ which contain lines from both classes.
    \end{lemma}

    \begin{proof}
        Two lines $\l_i \in \scL_i$ and $\l_j \in \scL_j$ intersect only if they are coplanar, so fix any such pair $\l_i$ and $\l_j$ and let $\Pi$ be the (unique) $2$-plane which contains them.
        Then $\Pi$ contains $q-1$ other lines from each of $\scL_i$ and $\scL_j$, since any parallel class in an affine $2$-plane over $\F_q$ has $q$ lines in it.
        Moreover, since any two non-parallel lines in an affine $2$-plane must intersect, these lines form a $K_{q,q}$ in the intersection graph between $\scL_i$ and $\scL_j$.
        Using the observation that all lines in $\scL_i$ which are not contained in $\Pi$ are necessarily skew to $\l_j$ and vice versa, it follows that these copies of $K_{q,q}$ are vertex-disjoint and partition the edge set of the intersection graph between $\scL_i$ and $\scL_j$.
        Finally, since the copies of $K_{q,q}$ are indexed by affine $2$-planes containing lines from both $\scL_i$ and $\scL_j$, we may count them by counting the number of such $2$-planes. In order for an affine $2$-plane to contain a line from $\scL_i$ and a line from $\scL_j$, this plane must be either $\Pi$ or an affine $2$-plane parallel to $\Pi$. The number of these $2$-planes is precisely $q^5/q^2 = q^3$. 
    \end{proof}

    Now, fix a subgraph $H \subseteq G$. 
    For a plane $\Pi$ and a set of lines $\scK$, we define $\scK(\Pi)$ to be the set of lines from $\scK$ which lie entirely in $\Pi$. For distinct parallel classes $\scL_i$ and $\scL_j$ in $\scL$ and a plane $\Pi$ containing lines from both $\scL_i$ and $\scL_j$, we define a bipartite graph
    $G_\Pi^{i,j}(H)$ as follows:
    \begin{itemize}[leftmargin=1.75em]
        \item the left vertex set is $\scL_i(\Pi)$, the $q$ lines from $\scL_i$ which lie in $\Pi$;
        \item the right vertex set is $\scL_j(\Pi)$, the $q$ lines from $\scL_j$ which lie in $\Pi$;
        \item $\l_i \in \scL_i(\Pi)$ and $\l_j \in \scL_j(\Pi)$ are adjacent if, for $x = \l_i \cap \l_j$ their unique common point,
        both incidences $x \in \l_i$ and $x \in \l_j$ remain edges in $H$.
    \end{itemize}
    Note that each such common point $x$ is well defined by the fact that two non-parallel lines in an affine plane must intersect. 
    Here comes the crucial point.

    \begin{lemma}
    \label{lem: no-rectangles}
        If $H$ is $C_8$-free, then every auxiliary graph $G_\Pi^{i,j}(H)$ is $C_4$-free.
    \end{lemma}

    \begin{proof}
        A $4$-cycle in $G_\Pi^{i,j}(H)$ consists of two lines in $\scL_i(\Pi)$ and two lines in $\scL_j(\Pi)$ with all
        four cross-intersections active. 
        These four lines then form a copy of $C_8$ in $H$, as the four intersection points
        are distinct because lines with the same direction are parallel and distinct. 
        Therefore, a $C_4$ in some
        $G_\Pi^{i,j}(H)$ would produce a $C_8$ in $H$.
    \end{proof}

    We are now ready to complete the proof of Theorem~\ref{thm: main}.
    Let $H\subseteq G$ be $C_8$-free and write $m=e(H)$. 
    The proof proceeds by double counting the total number $\Psi$ of cherries in $H$ consisting of a point and two lines.
    More precisely, let $\Psi$ be the number of triples $(x,\l_i,\l_j)$ where $x \in \F_q^5$, $\l_i \in \scL_i$, $\l_j \in \scL_j$ for distinct $i, j \in [q]$ and where $x = \l_i \cap \l_j$ with the incidences $(x,\l_i)$, $(x,\l_j)$ both being in $E(H)$. 

    On the one hand, $\Psi = \sum_{x \in \F_q^5} \binom{\deg_H(x)}{2}$. 
    We know  $\sum_{x \in \F_q^5} \deg_H(x) = m$, so the convexity of the function $t \mapsto \binom{t}{2}$ yields that
    \begin{equation}
    \label{eq: psi-one-hand}
        \Psi 
        = \sum_{x\in \F_q^5} \binom{\deg_H(x)}{2}
        \geq q^5\binom{m/q^5}{2}
        =\frac{m^2}{2q^5}-\frac{m}{2}.
    \end{equation}

    On the other hand, observe that the cherries we want to count are in one-to-one correspondence with edges in the auxiliary intersection graphs $G_\Pi^{i,j}(H)$, i.e., 
    \begin{equation}
    \label{eq: psi-other-hand}
        \Psi = \sum_{i \neq j \in [q]} \sum_\Pi |E\bigl(G_\Pi^{i,j}(H)\bigr)|,
    \end{equation}
    where the inside sum runs over the $q^3$ planes $\Pi$ which contain lines from both $\scL_i$ and $\scL_j$.
    Since $H$ is $C_8$-free, Lemmas~\ref{lem: intersection-graph} and~\ref{lem: no-rectangles}
    show that each $G_\Pi^{i,j}(H)$ is a $C_4$-free subgraph of $K_{q,q}$; therefore, the K\H{o}v\'ari--S\'os--Tur\'an theorem (see, e.g.,~\cite[Theorem VI.2.2]{Bollobas}) gives
    \begin{equation*}
        |E\bigl(G_\Pi^{i,j}(H)\bigr)| \leq q^{3/2}+q.
    \end{equation*}
    Summing over the $q^3$ planes and plugging back into~\eqref{eq: psi-other-hand}, we get that
    \begin{equation*}
        \Psi 
        \leq \sum_{i \neq j \in [q]} q^3 \cdot (q^{3/2}+q)
        = \binom{q}{2} \cdot q^3 \cdot (q^{3/2}+q)
        =O(q^{13/2}).
    \end{equation*}
    Combining this with~\eqref{eq: psi-one-hand}, it now follows that
    \begin{equation*}
        \frac{m^2}{2q^5}-\frac{m}{2} \leq \Psi =O(q^{13/2})
    \end{equation*}
    and so
    \begin{equation*}
        m=O(q^{23/4}).
    \end{equation*}
    In particular, $m=o(q^6)$. 
    Therefore, for every fixed $\d > 0$, the $C_8$-free subgraph $H$
    cannot have $\d q^6$ edges once $q$ is sufficiently large. 
\end{proof}

\vspace{-5mm} 

\section{Concluding remarks}
\label{section: conclusion}

Given that we deduce Theorem~\ref{thm: wenger} from the more general Theorem~\ref{thm: main} on the robustness of $C_8$'s in point-line incidence graphs over $\F_q^5$, it might be tempting to conjecture that all constructions of extremal $C_{10}$-free graphs share this property.
This is, however, not true: indeed, as stated in the introduction, the so-called {\it generalized hexagons}~\cite{LUW99} show that $\ex(n, \scC_{10}) = \Omega(n^{6/5})$. We thus believe that our result illustrates an interesting difference between the Wenger construction and other constructions of dense $C_{10}$-free graphs. 

By following the same steps as in the argument in Section \ref{section: wenger}, one may also show that for any $k \geq 2$ the incidence graph $G_\scL(q)$ between the points of $\F_q^k$ and $q$ parallel classes of lines has the property that
    \begin{equation*}
        \ex(G_\scL(q), H) = o(e(G_\scL(q)))
    \end{equation*}
for every fixed graph $H$ that is a subdivision of a bipartite graph. 
In particular, no edge-sampling argument in $W_{2k}(q)$ can produce a $C_{4k}$-free graph with $2q^{2k}$ vertices and $\Omega(q^{2k+1})$ edges for $k\geq 2$. 

Regarding more general pairs of even cycles, we believe that Conjecture~\ref{conj: main} should not hold for $C_{4k}$ and $C_{4k+2}$ for any $k \geq 2$. 
The two constructions in this paper confirm this for $k = 2$, but in general we expect that $\ex(Q_n,C_{4k}) = o(\ex(Q_n,C_{4k+2}))$, which, as in the proof of Theorem~\ref{thm: hypercube}, would yield the required counterexample. 
The intuition here, coming from~\cite{Con10}, is that the best known upper bound for $\ex(Q_n,C_{4k})$ is in a sense inherited from the upper bound for $\ex(N,C_{2k})$, while the upper bound for $\ex(Q_n,C_{4k+2})$ comes from the upper bound for $\ex(N,H)$ for an appropriate $3$-uniform hypergraph $H$. 
This then suggests that $\ex(Q_n,C_{4k})/e(Q_n)$ should drop off more quickly with $k$ than  $\ex(Q_n,C_{4k+2})/e(Q_n)$. 
If this is indeed the case, it is then likely that for any positive integer $t \equiv 2$ (mod $4$) there exists $k_0$ such that Conjecture~\ref{conj: main} does not hold for $C_{4k}$ and $C_{4k+t}$ for any $k \geq k_0$.


\end{document}